\documentclass[11pt]{article}
\usepackage{mathrsfs}
\usepackage{amssymb}
\usepackage{amsmath}
\usepackage{pict2e}
\usepackage{amsmath}
\usepackage{amsfonts}
\usepackage{graphicx}
\usepackage{xcolor}
\usepackage[small]{caption}
\usepackage{subfigure}
\usepackage{cite}

\def\0{\emptyset}

\def\q{\hfill\rule{1ex}{1ex}}

\begin{document}
\newtheorem{claim}{Claim}[section]
\newtheorem{theorem}{Theorem}[section]
\newtheorem{corollary}[theorem]{Corollary}
\newtheorem{definition}[theorem]{Definition}
\newtheorem{conjecture}[theorem]{Conjecture}
\newtheorem{question}[theorem]{Question}
\newtheorem{lemma}[theorem]{Lemma}
\newtheorem{proposition}[theorem]{Proposition}
\newenvironment{proof}{\noindent {\bf
Proof.}}{\rule{3mm}{3mm}\par\medskip}
\newcommand{\remark}{\medskip\par\noindent {\bf Remark.~~}}
\newcommand{\pp}{{\it p.}}
\newcommand{\de}{\em}

\title{\bf The $g$-good neighbor conditional diagnosability of locally exchanged twisted cubes}

\author{Huiqing Liu\footnote {Hubei Key Laboratory of Applied Mathematics, Faculty of Mathematics and Statistic, Hubei University, Wuhan 430062,
PR China}, Xiaolan Hu\footnote{School of Mathematics and Statistics $\&$ Hubei Key Laboratory of Mathematical Sciences, Central China Normal University,
Wuhan 430079, PR China} $^*$, Shan Gao$^*$
}

\date{}
\maketitle \baselineskip 16.5pt

\begin{abstract}

Connectivity and diagnosability are important parameters in measuring the fault tolerance and reliability of interconnection networks.
The $R^g$-vertex-connectivity of a connected graph $G$ is the minimum cardinality of a faulty set $X\subseteq V(G)$ such that $G-X$ is
disconnected and every fault-free vertex has at least $g$ fault-free neighbors. The $g$-good-neighbor conditional diagnosability is defined as the maximum cardinality of a $g$-good-neighbor conditional faulty set that the system can guarantee to identify.  The interconnection network considered here is the locally exchanged twisted cube $LeTQ(s,t)$. For $1\leq s\leq t$ and $0\leq g\leq s$,
we first determine the $R^g$-vertex-connectivity of $LeTQ(s,t)$, then establish the $g$-good neighbor conditional diagnosability of $LeTQ(s,t)$ under the PMC model and MM$^*$ model, respectively.

\vskip 0.1cm

\end{abstract}

\section{Introduction}

A multiprocessor system comprises two or more processors, and various processors exchange information via links between them. As the size of multiprocessor systems increase, processor failure is inevitable. When failure happens, we need to find the faulty processors to repair or replace them. Therefore, it is  crucial to tell the faulty processors from the good ones. The process of identifying faulty processors by analyzing the outcome of mutual tests among processors is called system-level diagnosis. There are several different diagnosis models being proposed to determine the diagnosability of a system. In this paper, we use the PMC model introduced by Preparata, Metzem and Chien \cite{Preparata} and MM$^*$ model suggested by Sengupta and Dahbura \cite{Sengupta}. In the PMC model, tests are performed between two adjacent processors. In the MM$^*$ model, every processor must test another two processors if it is adjacent to them.

For any processor in a system, it is impossible to determine whether it is fault-free or not if all its neighbors are faulty. Therefore, for a system, its diagnosability is restricted by its minimum degree. However, the probability of a faulty set containing all the neighbors of a processor is very low  in large multiprocessor systems. To obtain a more practical diagnosability,  Lai {\em et al.} \cite{Lai} introduced conditional diagnosability under the assumption that all the neighbors of any processor in a multiprocessor system cannot be faulty at the same time. The conditional diagnosability of interconnection networks has been extensively investigated, see \cite{Cheng,Hsieh1,Hsieh2,Hsu,Lin,Yang,ZhuLiu}.

Recently, Peng {\em et al.} \cite{Peng} proposed the $g$-good-neighbor conditional diagnosability, which requires that every fault-free node contains at least $g$ fault-free neighbors, and showed that the $g$-good-neighbor conditional diagnosability of $Q_n$ is $2^g(n-g+1)-1$ under the PMC model, where $0\leq g\leq n-3$. Since then, numerous studies have been investigated under the PMC model and/or MM$^*$ model. Wang and Han \cite{WangH} proved that the $g$-good-neighbor conditional diagnosability of $Q_n$ is $2^g(n-g+1)-1$ under the MM$^*$ model, where $0\leq g\leq n-3$.  Yuan {\em et al.} \cite{Yuan, Yuan2} established the $g$-good-neighbor diagnosability of the $k$-ary $n$-cubes under the PMC model and MM$^*$ model, respectively. Lin {\em et al.} \cite{LinXu} considered the $g$-good-neighbor diagnosability of the arrangement graphs under the PMC model and MM$^*$ model, respectively.
Xu {\em et al.} \cite{XuZhou} determined the $g$-good-neighbor conditional diagnosability of complete cubic networks. In \cite{Wei}, Wei and Xu studied the $g$-good-neighbor conditional diagnosabilities of the locally twisted cubes under the PMC and MM$^*$ model, respectively. Li and Lu \cite{Li} considered the $g$-good-neighbor diagnosability of star graphs under the PMC model and MM$^*$ model, respectively.
Cheng {\em et al.} \cite{ChengQiu} obtained the $g$-good-neighbor conditional diagnosabilities of the exchanged hypercube and its generalization under the PMC and MM$^*$ model, respectively. Very recently, Liu {\em et al.} \cite{Liu} determined the $g$-good neighbor conditional diagnosability of twisted hypercubes under the PMC and MM$^*$ model, respectively.

The interconnection network considered here is the locally exchanged twisted cube $LeTQ(s,t)$, which is a novel interconnection based on edge removal from the locally twisted cube $LTQ_{s+t+1}$. A major advantage is that it scales upward with lower edge costs than the locally twisted cube.
The topology is defined with two parameters, which provides more interconnection
flexibility. It maintains many desirable properties of the locally twisted cube, such as recursive
construction, partitionability, Hamiltonicity,  strong connectivity and super connectivity.  All these attractive properties of $LeTQ(s,t)$ make it applicable to large scale parallel computing systems very well.

Our main results are listed below.

\begin{theorem}
For $1\leq s\leq t$ and $0\leq g\leq s$. The $g$-good-neighbor conditional
diagnosability of $LeTQ(s,t)$ under the PMC model is $t_g(LeTQ(s,t))=2^g(s-g+2)-1$.
\end{theorem}

\begin{theorem}
For $1\leq s\leq t$ and $0\leq g\leq s$. The $g$-good-neighbor conditional
diagnosability of $LeTQ(s,t)$  under the MM$^*$ model is
$$t_g(LeTQ(s,t))=\left\{
\begin{array}{ll}
1& \mbox{if $0\leq g\leq 1$, $s=t=1$;}\\
s+1& \mbox{if $g=0$, $s+t\geq 3$;}\\
4& \mbox{if $g=1$, $s=2$ $\&$ $t\geq 2$;}\\
2s+1& \mbox{if $g=1$, $3\leq s\leq t$ or $s=1$ $\&$ $t\geq 2$;}\\
2^g(s-g+2)-1& \mbox{if $g\geq 2$, $2\leq s\leq t$.}
\end{array}
\right.$$
\end{theorem}

The rest of this paper is organized as follows: Section 2 provides preliminaries for our notations,  the locally exchanged twisted cubes and diagnosing a system. In Section 3, we determine the $R^g$-vertex-connectivity of locally exchanged twisted cubes. In Section 4, we establish the $g$-good-neighbor conditional diagnosability of locally exchanged twisted cubes under the PMC model and MM$^*$ model, respectively. Our conclusions are given in Section 5.

\section{Preliminaries}

\subsection{Notations}

Let $G=(V(G),E(G))$ be a simple and finite graph. The {\em neighborhood} $N_G(v)$ of a vertex $v$ is the set of vertices adjacent to $v$ and the {\em closed neighborhood} of $v$ is $N_G[v]=N_G(v)\cup \{v\}$. The {\em degree} $d_G(v)$ of $v$ is $|N_G(v)|$.  The {\em minimum  degree} of $G$ is denoted by $\delta(G)$. If $d_G(v)=k$ for any $v\in V(G)$, then $G$ is called a {\em $k$-regular} graph. For $S\subseteq V(G)$, $G[S]$ denotes the subgraph induced by $S$. The {\em neighborhood set} of $S$ is defined as $N_G(S)=(\cup_{v\in S}N_G(v))-S$, and the {\em closed neighborhood set} of $S$ is defined as $N_G[S]=N_G(S)\cup S$. We will use $G-S$ to denote the subgraph $G[V(G) -S]$. For any $v\in V(G)$, $N_S(v)$ denotes the neighborhood of $v$ in $S$. For two disjoint subsets $S,T$ of $V(G)$, let $E_G(S,T)=\{uv\in E(G)~|~u\in S,~v\in T\}$. The {\em symmetric difference} of $F_1\subseteq V(G)$ and $F_2\subseteq V(G)$ is defined as the set $F_1\triangle F_2=(F_1-F_2) \cup (F_2-F_1)$.

The minimum cardinality of a vertex set $S\subseteq V(G)$ such that  $G-S$ is disconnected or has only one vertex, denoted by $\kappa(G)$,
is the {\em connectivity} of $G$. A subset $F\subseteq V(G)$ is called an {\em $R^g$-vertex-set} of $G$ if $\delta(G-F)\geq g$. An {\em $R^g$-vertex-cut} of a connected graph $G$ is a $R^g$-vertex-set $F$ such that $G-F$ is disconnected. The {\em $R^g$-vertex-connectivity} of $G$, denoted by $\kappa^g(G)$, is the cardinality of a minimum $R^g$-vertex-cut of $G$. Note that $\kappa^0(G)=\kappa(G)$.

\subsection{The locally exchanged twisted cubes}

In this subsection, we first give the definition of locally twisted cubes and locally exchanged twisted cubes, respectively, and then present some properties of locally exchanged twisted cubes.

Let ``$\oplus$" represent the modulo 2 addition.
Let $u_{n-1}u_{n-2}\cdots u_1u_0$ be an $n$-bit binary string.  The complement of $u_i$ in $\{0, 1\}$ will be
denoted by $\overline{u_i}$ ($\overline{0}=1$ and $\overline{1}=0$). Yang {\em et al.} \cite{YME} proposed the following non-recursive definition of $LTQ_n$.

\begin{definition}\cite{YME}
\label{LTQ}
Let $n$ be a positive integer. The locally twisted cube $LTQ_n$ of dimension $n$ has $2^n$ vertices, each labeled by
an $n$-bit binary string $u_{n-1}u_{n-2}\cdots u_1u_0$. Any two vertices $u =u_{n-1}u_{n-2}\cdots u_1u_0$ and $v = v_{n-1}v_{n-2}\cdots v_1v_0$ of $LTQ_n$ are adjacent if and only if one of the following conditions is satisfied:

\noindent (1) There is an integer $k$ ($2 \leq k \leq n-1$) such that $u_k = \overline{v_k}$,
 $u_{k-1} = v_{k-1}\oplus u_0$, and all the remaining bits of $u$ and $v$ are identical;

\noindent (2) There is an integer $k$ ($0\leq k \leq 1$) such that $u_k = \overline{v_k}$, and all the remaining bits of $u$ and $v$ are identical.
\end{definition}

As a variant of hypercubes, the locally twisted cube $LTQ_n$ preserves many of its desirable properties such as regularity, Hamiltonicity, strong
connectivity and high recursive constructability. Moreover, $LTQ_n$ also keeps a nice property of $Q_n$, that is, any two adjacent vertices in $LTQ_n$ differ only in at most two successive bits. However, the diameter of $LTQ_n$ is only about half of that of $Q_n$. Furthermore, $LTQ_n$ is superior to $Q_n$ in cycle embedding property as $LTQ_n$ contains cycles of all lengths from 4 to $2^n$ \cite{YME}, but $Q_n$ contains only even cycles since it is a bipartite graph.

Recently, Chang {\em et al.} \cite{CCYW} proposed the definition of locally exchanged twisted cube, which not only kept numerous desirable properties of the locally twisted cube, but also reduced the interconnection complexity.

\begin{definition}\cite{CCYW}
\label{LeTQ}
A locally exchanged twisted cube $LeTQ(s, t)$  with integers $s, t\geq 1$ is defined as an undirected graph. The vertex set $V=\{a_{s-1}\cdots a_1a_0b_{t-1}\cdots b_1b_0c~|~a_i,b_j,c\in \{0,1\}, 0\leq i\leq s-1, 0\leq j\leq t-1\}$. Two vertices $u=a_{s-1}\cdots a_1a_0b_{t-1}\cdots b_1b_0c$ and $v=a'_{s-1}\cdots a'_1a'_0b'_{t-1}\cdots b'_1b'_0c'$ are adjacent if and only if one of the following conditions is satisfied:

\noindent (1) $\overline{c}=c'$, and all the remaining bits of $u$ and $v$ are identical;

\noindent (2) $c=c'=1$, and one of the following conditions is satisfied:

(a) There is an integer $k$ ($0\leq k\leq 1$) such that $\overline{b_k}= b'_k$,  and all the remaining bits of $u$ and $v$ are identical,

(b) $b_0=b'_0=1$, there is an integer $k$ ($2\leq k\leq t-1$) such that $\overline{b_k}= b'_k$ and $\overline{b_{k-1}}= b'_{k-1}$, and all the remaining bits of $u$ and $v$ are identical,

(c) $b_0=b'_0=0$, there is an integer $k$ ($2\leq k\leq t-1$) such that $\overline{b_k}= b'_k$, and all the remaining bits of $u$ and $v$ are identical;

\noindent (3) $c=c'=0$, and one of the following conditions is satisfied:

(a) There is an integer $k$ ($0\leq k\leq 1$) such that $\overline{a_k}= a'_k$  and all the remaining bits of $u$ and $v$ are identical,

(b) $a_0=a'_0=1$, there is an integer $k$ ($2\leq k\leq s-1$) such that $\overline{a_k}= a'_k$ and $\overline{a_{k-1}}= a'_{k-1}$, and all the remaining bits of $u$ and $v$ are identical,

(c) $a_0=a'_0=0$, there is an integer $k$ ($2\leq k\leq s-1$) such that $\overline{a_k}= a'_k$, and all the remaining bits of $u$ and $v$ are identical.
\end{definition}

%%%%%%%%%%%%%%%%%%%%%%%?7?
According to Definition~\ref{LeTQ}, Figure 1 illustrates $LeTQ(1,1)$ and  $LeTQ(1,2)$.

\begin{center}
\setlength{\unitlength}{1.0mm}
\begin{picture}(110,37)
\thinlines
\scriptsize

\textcolor{blue}{\put(15,5){\circle*{1}}}   \put(25,5){\circle*{1}}

\textcolor{blue}{\put(15,15){\circle*{1}}}  \put(25,15){\circle*{1}}

\textcolor{blue}{\put(15,25){\circle*{1}}}  \put(25,25){\circle*{1}}

\textcolor{blue}{\put(15,35){\circle*{1}}}  \put(25,35){\circle*{1}}

\put(15,5){\line(1,0){10}} \put(15,35){\line(1,0){10}}

\put(15,5){\line(0,1){10}} \put(15,25){\line(0,1){10}} \put(25,5){\line(0,1){10}} \put(25,25){\line(0,1){10}}

\put(15,15){\line(10,10){10}} \put(25,15){\line(-10,10){10}}

\put(8,4){110}\put(8,14){010}\put(8,24){100}\put(8,34){000}

\put(27,4){111}\put(27,14){101}\put(27,24){011}\put(27,34){001}

%%%%%%%%%%%%%%%%%%%%%%%

\textcolor{blue}{ \put(56,5){\circle*{1}}} \textcolor{blue}{\put(55,15){\circle*{1}}} \textcolor{blue}{ \put(54,25){\circle*{1}}} \textcolor{blue}{\put(53,35){\circle*{1}}}

\put(63,5){\circle*{1}}\put(63,15){\circle*{1}}\put(63,25){\circle*{1}} \put(63,35){\circle*{1}}

\put(83,5){\circle*{1}}\put(83,15){\circle*{1}}\put(83,25){\circle*{1}} \put(83,35){\circle*{1}}

\textcolor{blue}{ \put(93,5){\circle*{1}}} \textcolor{blue}{ \put(92,15){\circle*{1}}} \textcolor{blue}{ \put(91,25){\circle*{1}}} \textcolor{blue}{\put(90,35){\circle*{1}}}

\put(50,5){\line(1,0){10}}\put(50,35){\line(1,0){10}}  \put(80,5){\line(1,0){10}} \textcolor{red}{\put(60,15){\line(1,0){20}}}

\textcolor{red}{\put(59,5){\line(1,0){20}}}  \textcolor{red}{ \put(58,35){\line(1,0){20}}} \put(78,35){\line(1,0){10}}
\textcolor{red}{\put(58,25){\line(1,0){20}}}

\put(48,5){\line(0,1){10}}\put(58,5){\line(0,1){10}}\put(78,5){\line(0,1){10}}\put(88,5){\line(0,1){10}}

\put(48,25){\line(0,1){10}}\put(58,25){\line(0,1){10}}\put(78,25){\line(0,1){10}}\put(88,25){\line(0,1){10}}

\put(48,15){\line(10,10){10}} \put(58,15){\line(-10,10){10}}\put(78,15){\line(10,10){10}} \put(88,15){\line(-10,10){10}}

\put(40,5){1100}\put(40,15){0100}\put(40,25){1000}\put(40,35){0000}

\put(60,6){1101}\put(60,16){1001}\put(60,26){0101}\put(60,36){0001}

\put(70,6){1111}\put(70,16){1011}\put(70,26){0111}\put(70,36){0011}

\put(90,5){1110}\put(90,15){0110}\put(90,25){1010}\put(90,35){0010}

%\normalsize

 \put(6,-3){$LeTQ(1,1)$}    \put(64,-3){$LeTQ(1,2)$}

\normalsize
\put(-12,-10){Figure 1~~Locally exchanged twisted cubes $LeTQ(1,1)$ and $LeTQ(1,2)$}

\end{picture}
\end{center}

\vskip 1.0cm

%%%%%%%%%%%%%%%%%%%%%%%%%%%%%%%%%%%%%%%%%%%%%%%%%%%%%%%%%%%%%%%%%%%%%%%%%%%%%%%%%%%%%%%%%%%%%%%%%%%%%%%%%%%%%%%%%%%%%%%%%%%%%%%%%%%%%
Let $LeTQ_{x}^i(s, t)$ be the subgraph of $LeTQ(s, t)$ by fixing $x=i$ for $i\in \{0,1\}$ and $x\in\{a_0,a_1,\ldots,s_{s-1}\}\cup\{b_0,b_1,\ldots,b_{s-1}\}$. By the definition of $LeTQ(s, t)$, Chang {\em et al.} \cite{CCYW} proposed the following two propositions.

\begin{proposition} \cite{CCYW}
\label{two-sub}
If $s\geq 2$, $LeTQ(s, t)$ can be decomposed into two subgraphs $LeTQ_{a_i}^0(s, t)$ and $LeTQ_{a_i}^1(s, t)$, which are isomorphic to $LeTQ(s-1, t)$ by fixing $a_i$ for $0\leq i\leq s-1$, and if $t\geq 2$, $LeTQ(s, t)$ can be decomposed into two subgraphs $LeTQ_{b_j}^0(s, t)$ and $LeTQ_{b_j}^1(s, t)$, which are isomorphic to $LeTQ(s, t-1)$ by fixing $b_j$ for $0\leq j\leq t-1$. Furthermore, there are $2^{s+t-1}$ independent edges between $LeTQ_{x}^0(s, t)$ and $LeTQ_{x}^1(s, t)$, $x\in \{a_0,a_1,\ldots,a_{s-1},b_0,b_1,\ldots,b_{t-1}\}$.
\end{proposition}

\begin{proposition} \cite{CCYW}
\label{cong}
$LeTQ(s, t)\cong LeTQ(t, s)$.
\end{proposition}

By Proposition~\ref{cong}, without loss of generality, we always assume that $s\leq t$.

We can partition $V(LeTQ(s,t))$ into  $L$ and $R$, in which

~~~~~~$L=\{a_{s-1}\cdots a_1a_0b_{t-1}\cdots b_1b_00~|~a_i,b_j\in \{0,1\}, 0\leq i\leq s-1, 0\leq j\leq t-1\}$,

~~~~~~$R=\{a_{s-1}\cdots a_1a_0b_{t-1}\cdots b_1b_01~|~a_i,b_j\in \{0,1\}, 0\leq i\leq s-1, 0\leq j\leq t-1\}$.

\noindent  For any $u=a_{s-1}\cdots a_1a_0b_{t-1}\cdots b_1b_0c$, we denote $u=A(u)B(u)C(u)$ for simplicity, where $A(u)=a_{s-1}\cdots a_1a_0$, $B(u)=b_{t-1}\cdots b_1b_0$ and $C(u)=c$.
We can partition $L$ into $L_i$ ($1\leq i\leq 2^t$) such that for two vertices $u=A(u)B(u)0$ and $v=A(v)B(v)0$ of $L_i$,  $B(u)=B(v)$.
Then $|L_i|=2^s$. Similarly, we can partition $R$ into $R_j$ ($1\leq j\leq 2^s$) such that for two vertices $u=A(u)B(u)1$ and $v=A(v)B(v)1$  of $R_j$,  $A(u)=A(v)$. Then $|R_j|=2^t$.

By the definitions of $LTQ_n$ and $LeTQ(s,t)$, we have:

\begin{proposition}
\label{LR1}

\noindent (1) Each subgraph induced by $L_i$ ($1\leq i\leq 2^t$)  is a $LTQ_s$ and we call this subgraph a Class-0 cluster, each subgraph induced by $R_j$ ($1\leq i\leq 2^s$) is a $LTQ_t$ and  we call this subgraph a Class-1 cluster;

\noindent (2) There are no edges between $L_i$ and $L_k$ for $1\leq i,k\leq 2^t$ and $i\neq k$, and there are no edges between $R_j$ and $R_k$ for $1\leq j,k\leq 2^s$ and $j\neq k$.
\end{proposition}

Note that for any $u=a_{s-1}\cdots a_1a_0b_{t-1}\cdots b_1b_0c$ in $L$ (or resp., $R$), $u$ has a unique neighbor $v=a_{s-1}\cdots a_1a_0b_{t-1}\cdots b_1b_0\overline{c}$ in $R$ (or resp., $L$). In this case, we call $uv$ a {\it cross edge} and denote $v=u^{*}$. By Proposition~\ref{LR1}, we have the following proposition.

\begin{proposition}
\label{degree}
Each vertex in $L$ has degree $s+1$, and each vertex in $R$ has degree $t+1$.
\end{proposition}

Let $u=A(u)B(u)0$ and $w=A(w)B(w)0$ be two vertices of a Class-0 cluster, then $A(u)\neq A(w)$ and $B(u)=B(w)$. Note that $u^{*}=A(u)B(u)1$ and $w^{*}=A(w)B(w)1$, then $u^*$ and $w^*$ belong to two different Class-1 clusters. Similarly, if $u$ and $w$ are two vertices of a Class-1 cluster, then $u^{*}$ and $w^{*}$ belong to two different Class-0 clusters. So we have the following proposition.

\begin{proposition}
\label{LR2}
Each vertex in $L$ has a unique neighbor in $R$ and vice visa. Furthermore, for two different vertices $u,v$ of a Class-$c$ cluster, $u^{*}$ and $v^{*}$ belong to two different Class-$\overline{c}$ clusters, where $c\in \{0,1\}$.
\end{proposition}

\begin{proposition}
\label{common neighbor}
$LeTQ(s,t)$ contains no triangles. Furthermore, $LeTQ(s,t)$ contains no $K_{2,3}$. That is, for any  $u,v\in V(LeTQ(s,t))$, $u$ and $v$ have at most two common neighbors.
\end{proposition}

\begin{proof}
We proof this proposition by induction on $s+t$. In the basis step, for $s=t=1$, $LeTQ(1,1)$ is a 8-cycle. It is seen that $LeTQ(1,1)$ contains neither triangles nor $K_{2,3}$. In the induction step, assume the statement is true for $s+t=k-1$ with $k\geq 3$. Then we consider the case of $s+t=k$. Since $s\le t$, $t\ge 2$. By Proposition~\ref{two-sub}, $LeTQ(s, t)$ can be decomposed into two subgraphs  $LeTQ^0(s, t)$ and $LeTQ^1(s, t)$ which are isomorphic to $LeTQ(s, t-1)$. By induction hypothesis, $LeTQ^i(s, t)$ contains neither triangles nor $K_{2,3}$ for $i=0,1$. Since there are exactly $2^{s+t-1}$ independent edges between $LeTQ^0(s, t)$ and $LeTQ^1(s, t)$, $LeTQ(s,t)$ contains no triangles, and for $u,v\in V(LeTQ^0(s,t))$ or $u,v\in V(LeTQ^1(s,t))$, $u$ and $v$ have at most two common neighbors. Note that for $u\in V(LeTQ^0(s,t))$ and $v\in V(LeTQ^1(s,t))$, $u$ has at most one neighbor in $V(LeTQ^1(s,t))$ and $v$ has at most one neighbor in $V(LeTQ^0(s,t))$, and hence $u$ and $v$ have at most two common neighbors.
\end{proof}

\begin{proposition}
\label{LeTQ-neighbor}
Let $H$ be a subgraph of $LeTQ(s,t)$. If $\delta(H)\geq g$, then $|V(H)| \geq 2^g$.
\end{proposition}

\begin{proof}
We proof this proposition by induction on $s+t$. In the basis step, for $s=t=1$, $LeTQ(1,1)$ is a 8-cycle. Since $\delta(H)\geq g$, $0\leq g\leq 2$. If $0\leq g\leq 1$, then $K_{g+1}$ is a subgraph of $H$, and thus $|V(H)| \geq 2^g$. If $g=2$, then $H=LeTQ(1,1)$, and thus $|V(H)|=8>2^2$. Hence, the basic step holds. In the induction step, assume the statement is true for $s+t=k-1$ with $k\geq 3$. Then we consider the case of $s+t=k$. Since $s\le t$, $t\ge 2$. Let $H$ be a subgraph of $LeTQ(s,t)$ with $\delta(H)\geq g$.
By Proposition~\ref{two-sub}, $LeTQ(s, t)$ can be decomposed into two subgraphs  $LeTQ^0(s, t)$ and $LeTQ^1(s, t)$ which are isomorphic to $LeTQ(s, t-1)$.
Let $H_0$ be the subgraph induced by $V(H)\cap V(LeTQ^0(s, t))$ and $H_1$ be the subgraph induced by $V(H)\cap V(LeTQ^1(s, t))$.
%We consider the following two cases.

\vskip 0.2cm
\noindent{\bf Case 1.} $|V(H_0)|=0$ or $|V(H_1)|=0$.
\vskip 0.2cm

Without loss of generality, assume that $|V(H_0)|=0$. That is, $H=H_1$ is a subgraph of $LeTQ^1(s, t)$. Note that $\delta(H)\geq g$, then $|V(H)| \geq 2^g$ by induction hypothesis.

\vskip 0.2cm
\noindent{\bf Case 2.} $|V(H_0)|>0$ and $|V(H_1)|>0$.
\vskip 0.2cm

Note that $\delta(H)\geq g$ and each vertex in $LeTQ^i(s, t)$ has at most one neighbor in $LeTQ^{\overline{i}}(s, t)$, then $\delta(H_i)\geq g-1$, and thus $|V(H_i)|\geq 2^{g-1}$ by induction hypothesis for $i=0,1$. Hence $|V(H)|=|V(H_0)|+|V(H_1)|\geq 2^{g-1}+2^{g-1}=2^g$.
\end{proof}

\subsection{The PMC model and MM$^*$ model for diagnosis}

A multiprocessor system is typically represented by an undirected simple graph $G=(V, E)$, where $V(G)$ stands for the processors and $E(G)$ represents the link between two processors.

Preparata {\em et al.} \cite{Preparata} proposed the PMC model, which performs diagnosis by testing the neighboring processor via the links between them. Under the PMC model, tests can be performed between any two adjacent vertices $u$ and $v$. We use the ordered pair $(u, v)$  to denote a test that $u$ diagnoses $v$.  The result of a test $(u, v)$ is reliable if and only if $u$ is fault-free. In this condition, the result is 0 if $v$ is fault-free, and is 1 otherwise.

A test assignment for a system $G$ is a collection of tests, which can be represented a directed graph $T = (V(G), L)$, where $V(G)$ is the vertex set of $G$ and $L=\{(u,v)~|~u,v\in V(G)$ and $uv\in E(G)\}$.  The collection of all test results from the test scheme $T$ is termed as a syndrome $\sigma: L\rightarrow \{0, 1\}$. Let $T = (V (G), L)$ be a test assignment, and $F$ a subset of $V (G)$. For any given syndrome $\sigma$ resulting from $T$, $F$ is said to be {\em consistent} with $\sigma$ if the syndrome $\sigma$ can be produced when all vertices in $F$ are faulty and all vertices in $V(G)- F$ are  fault-free. That is, if $u$ is fault-free, then $\sigma((u,v))=0$ if $v$ is fault-free, and $\sigma((u,v))=1$ if $v$ is faulty; if $u$ is faulty,
 then $\sigma((u,v))$ can be either 0 or 1 no matter $v$ is faulty or not. Therefore, a faulty set $F$ may be consistent with different syndromes.
 We use $\sigma(F)$ to represent the set of all possible syndromes with which the faulty set $F$ can be consistent.

Let $F_1$ and $F_2$ be two distinct faulty sets of $V (G)$, $F_1$ and $F_2$ are {\em distinguishable} if $\sigma(F_1)\cap \sigma(F_2) = \emptyset$; otherwise, $F_1$ and $F_2$ are {\em indistinguishable}. In other words, if $\sigma(F_1)\cap \sigma(F_2) = \emptyset$, then $(F_1, F_2)$ is a {\em distinguishable pair}; otherwise, $(F_1, F_2)$ is an {\em indistinguishable pair}. A system $G$ is called {\em $t$-diagnosable}, if any two distinct faulty sets $F_1,F_2\subseteq V(G)$ are distinguishable, provided that $|F_1|,|F_2|\leq t$.
% A system $G$ is called {\em $t$-diagnosable} if for any two distinct faulty sets $F_1,F_2\subseteq V(G)$ with $|F_1|,|F_2|\leq t$ are distinguishable.
The {\em diagnosability} of $G$, denoted as $t(G)$, is the maximum $t$ such that $G$ is $t$-diagnosable.

Dahbura and  Masson \cite{Dahbura} proposed a sufficient and necessary condition of $t$-diagnosable systems.

\begin{proposition} \cite{Dahbura}
\label{diagnosable}
A system $G =(V, E)$ is $t$-diagnosable if and only if, for any two distinct subsets $F_1$ and $F_2$ of $V(G)$ with $|F_1|,|F_2|\leq t$, there is at least one test from $V(G)-F_1-F_2$ to $F_1\triangle F_2$.
\end{proposition}

For the PMC model, Peng {\em et al.} \cite{Peng} proposed  the following proposition.

\begin{proposition} \cite{Peng}
\label{PMC}
For any two distinct subsets $F_1$ and $F_2$ of $V(G)$, $(F_1, F_2)$ is a distinguishable pair under the PMC model if and only if there is a vertex $u\in V(G)-F_1-F_2$ and a vertex $v\in F_1\triangle F_2$ such that $(u, v) \in E(G)$ (see Figure 2).
\end{proposition}

\begin{figure}[!htb]
\centering
{\includegraphics[height=0.15\textwidth]{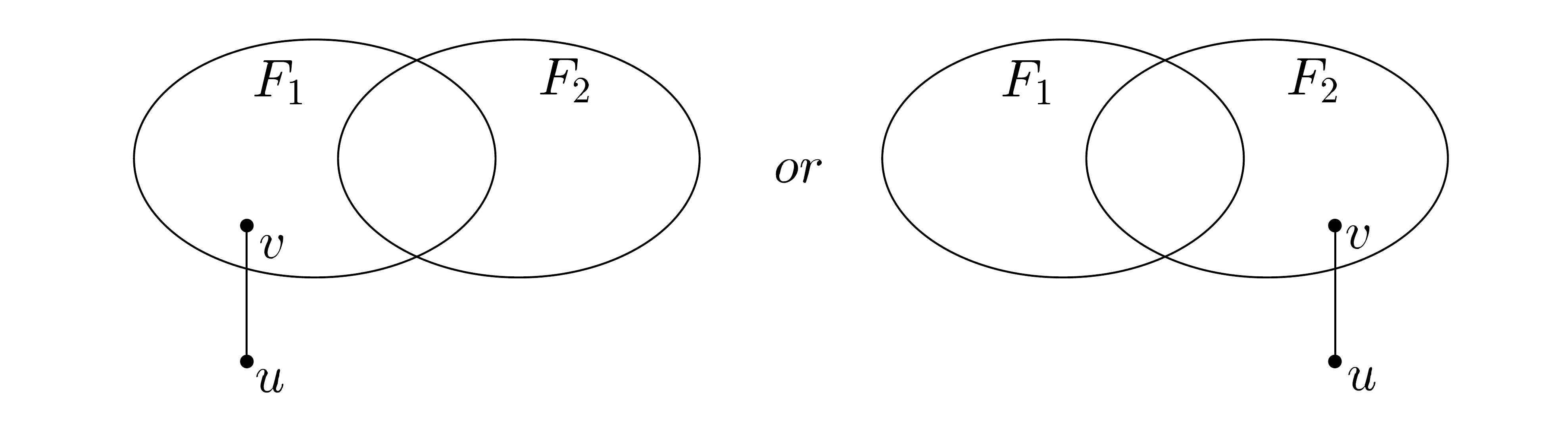}}

Figure 2 Illustration of a distinguishable pair under the PMC model
\end{figure}

Sengupta and Dahbura \cite{Sengupta} proposed the MM$^*$ model, which performs diagnosis by sending the same inputs to a pair of
adjacent processors, and comparing their responses. In the MM$^*$ model, tests can be performed from any vertex $w$ to its any two neighbors $u$ and $v$. We use the labeled pair $(u, v)_w$ to denote a test performed by $w$ on its neighbors $u$ and $v$. The result of a test $(u, v)_w$ is reliable if and only if $w$ is fault-free. In this condition, the result is 0 if both $u$ and $v$ are fault-free, and is 1 otherwise.

The test scheme of a system $G$ is often represented by a multigraph $M=(V(G), L)$, where $L =\{(u, v)_w~|~u, v, w \in V(G)$ and $uw, vw \in E(G)\}$. Since a pair of vertices may be compared by different vertices, $M$ is a multigraph. The collection of all test results from the test scheme $M$ is termed as a syndrome $\sigma: L\rightarrow \{0, 1\}$. Given a subset of vertices $F\subseteq V(G)$, we say that $F$ is consistent with $\sigma$ if the syndrome $\sigma$ can be produced when all nodes in $F$ are faulty and all nodes in $V(G)-F$ are fault-free. That is, if $w$ is fault-free, then $\sigma((u, v)_w) =0$ if $u$ and $v$ are fault-free, and  $\sigma((u, v)_w) =1$  otherwise; if $w$ is faulty, then  $\sigma((u, v)_w)$ can be either 0 or 1 no matter $u$ and $v$ are faulty or not.

Sengupta and Dahbura \cite{Sengupta} proposed a sufficient and necessary condition for two distinct subsets  $F_1$ and $F_2$ to be a distinguishable
pair under the MM$^*$ model.

\begin{proposition}\cite{Sengupta}
\label{MM}
For any two distinct sets $F_1,F_2\subseteq V(G)$, $(F_1, F_2)$ is a distinguishable pair under the MM$^*$ model if and only if one of the following conditions is satisfied (see Figure 3):

\noindent(1) There are two vertices $u, w\in V(G)-F_1-F_2$ and there is a vertex $v \in F_1\triangle F_2$ such that $(u, w)$, $(v, w)\in E(G)$;

\noindent(2) There are two vertices $u, v\in F_1-F_2$ and there is a vertex $w \in V(G)-F_1-F_2$ such that $(u, w)$,  $(v, w)\in E(G)$;

\noindent(3) There are two vertices $u, v \in F_2-F_1$ and there is a vertex $w\in V(G)-F_1-F_2$ such that  $(u, w)$, $(v, w)\in E(G)$.
\end{proposition}

\begin{figure}[!htb]
\centering
{\includegraphics[height=0.25\textwidth]{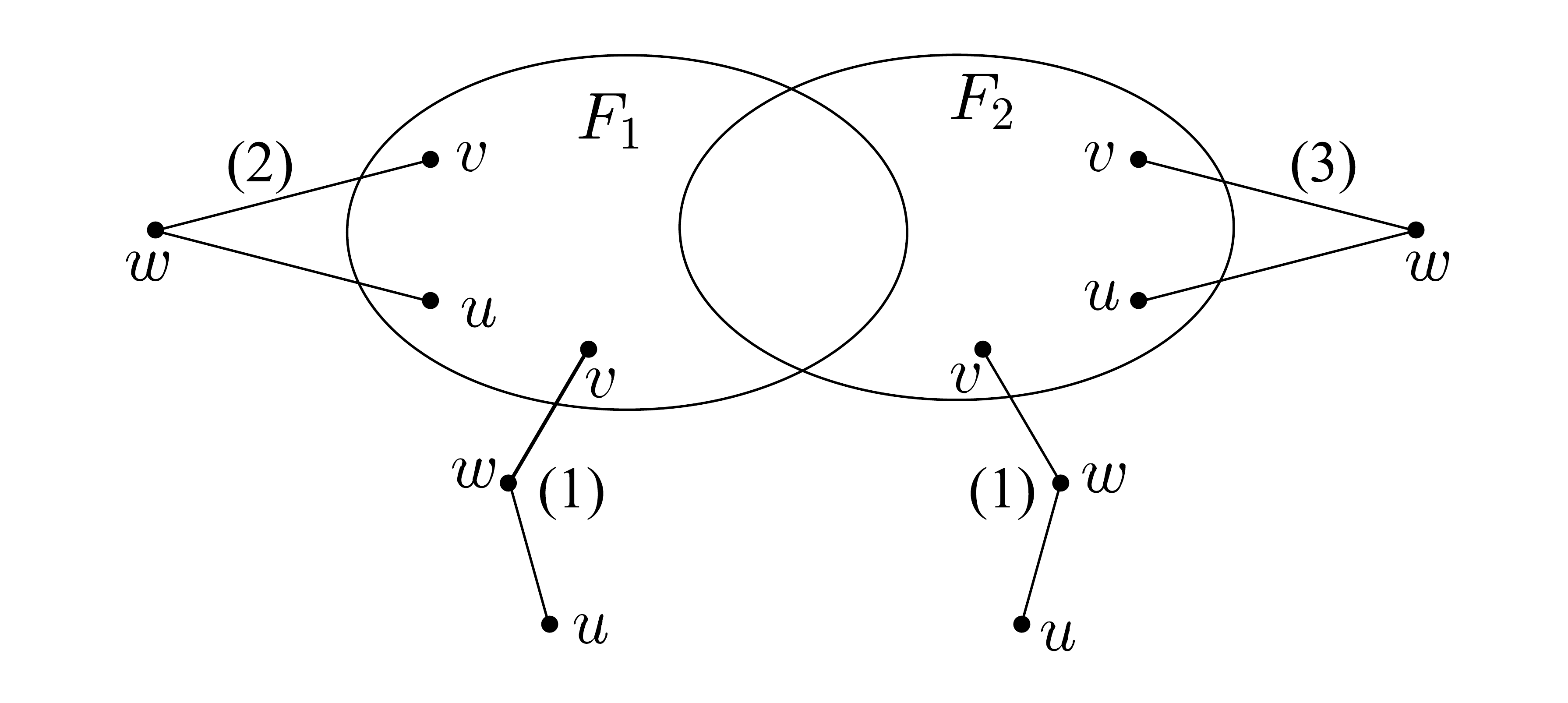}}

Figure 3 Illustration of a distinguishable pair under the MM$^*$ model
\end{figure}

In the following, we will give the definition of $g$-good-neighbor conditional faulty set and the $g$-good-neighbor conditional diagnosability of a system.

\begin{definition}
Let $G = (V, E)$ be a system. A faulty set $F\subseteq V(G)$ is called a $g$-good-neighbor conditional faulty set if $\delta(G-F)\geq g$.
\end{definition}

\begin{definition}
%A system $G = (V, E)$ is $g$-good-neighbor conditional $t$-diagnosable if for any two distinct $g$-good-neighbor conditional faulty sets $F_1, F_2\subseteq V(G)$ with $|F_1|, |F_2|\leq t$ are distinguishable.
A system $G = (V, E)$ is $g$-good-neighbor conditional $t$-diagnosable, if any two distinct $g$-good-neighbor conditional faulty sets $F_1, F_2\subseteq V(G)$ are distinguishable, provided that $|F_1|, |F_2|\leq t$. The $g$-good-neighbor conditional diagnosability of $G$, denoted by $t_g(G)$, is the maximum $t$ such that $G$ is $g$-good-neighbor conditionally $t$-diagnosable.
\end{definition}

\section{The $R^g$-vertex-connectivity of $LeTQ(s,t)$}

In order to obtain the $g$-good-neighbor conditional diagnosability of $LeTQ(s,t)$, we first investigate the $R^g$-vertex-connectivity of $LeTQ(s,t)$, which is closely related to $g$-good-neighbor conditional diagnosability proposed by Latifi \cite{Latifi}.  As a more refined index than the traditional connectivity, the $R^g$-vertex-connectivity can be used to measure the conditional fault tolerance of networks. In this section, we determine the $R^g$-vertex-connectivity of $LeTQ(s,t)$ for $1\leq s\leq t$ and $0\leq g\leq s$.

For the sake of simplicity, we always use $V$ to denote the vertex set of $LeTQ(s,t)$  in the following discussion.
For any $S\subseteq V(LeTQ(s,t))$, we use $N_V(S)$ and $N_V[S]$ to denote the neighborhood set and  closed neighborhood set of $S$ in $LeTQ(s,t)$.

For convenience, we write the consecutive $k$ 0's in the binary string as $0^k$.

\begin{lemma}%Lemma 3.1
\label{g-neighbor}
For $1\leq s \leq t$ and $0\leq g\leq s$, let $A=\{a_{s-1}a_{s-2}\cdots a_{s-g} 0^{s-g+t+1}~|~a_i\in \{0,1\},~s-g\leq i\leq s-1\}\subset V$, $F_1=N_{V}(A)$ and $F_2=N_{V}[A]$. Then $|F_1|=2^g(s-g+1)$ and $|F_2|=2^g(s-g+2)$. Furthermore, $F_1$ is a $g$-good-neighbor conditional faulty set and $F_2$ is a $\max\{s-1,g\}$-good-neighbor conditional faulty set of $LeTQ(s,t)$.
\end{lemma}

\begin{proof}
First we consider the case of $g=s$.  By Proposition~\ref{LR1}, the subgraph induced by $A$ is a $LTQ_s$, then $|F_1|=|A|=2^s=2^g(s-g+1)$. Note that $A\subseteq L$ and $F_1\subseteq R$, $|F_2|=|F_1|+|A|=2^s+2^s=2^{s+1}=2^g(s-g+2)$. Let $u$ be any vertex of $V-F_2$. If $u$ belongs to a  Class-0 cluster, then all the neighbors of $u$ are out of $F_2$ by Propositions~\ref{LR1} and \ref{LR2}, i.e., $u$  has $s+1$ neighbors out of $F_2$ by Proposition~\ref{degree}. If $u$ belongs to a  Class-1 cluster, then $u$ has at most one neighbor in $F_1$ as any two vertices of $F_1$ belong to different Class-1 clusters, i.e., $u$  has $t\geq s$ neighbors out of $F_2$. Thus $\delta(V-F_2)\geq s$ and hence $F_2$ is a $s$-good-neighbor conditional faulty set of $LeTQ(s,t)$.
Note that the subgraph induced by $A$ is a $LTQ_s$, then $\delta(V-F_1)\geq s$. Thus $F_1$ is a $s$-good-neighbor conditional faulty set of $LeTQ(s,t)$.

Then in the following, we consider the case of $g\leq s-1$.
Let $v=a_{s-1}a_{s-2}\cdots a_{s-g} 0^{s-g+t+1}\in A$. Since $s-g\geq 1$, $v$ has $g$ neighbors $\{a_{s-1}\cdots a_{s-g+1}\overline{a_{s-g}} 0^{s-g+t+1}, a_{s-1}\cdots a_{s-g+2}\overline{a_{s-g+1}}a_{s-g}0^{s-g+t+1},$ $ \ldots, a_{s-1}\overline{a_{s-2}}a_{s-3}\cdots a_{s-g}0^{s-g+t+1}$, $\overline{a_{s-1}}a_{s-2}\cdots a_{s-g}0^{s-g+t+1}\}$ in $A$,  and thus the subgraph induced by $A$ is a $g$-regular graph. Denote $A^*=\{a_{s-1}a_{s-2}\cdots a_{s-g}0^{s-g+t}1~|~a_i\in \{0,1\},~s-g\leq i\leq s-1\}$, then any two vertices of $A^*$ belong to different Class-1 clusters by Propositions~\ref{LR2}. Since $F_1=N_{V}(A)$, we have
$$F_1=\left(\cup_{0\leq j\leq s-g-1}\{a_{s-1}a_{s-2}\cdots a_{s-g}0^{s-g-j-1}10^{j+t+1}~|~a_i\in \{0,1\},~s-g\leq i\leq s-1\}\right)\bigcup A^*.$$ Then  $|F_1|=2^g(s-g+1)$. Note that $F_1\cap A=\emptyset$, and hence $|F_2|=|F_1|+|A|=2^g(s-g+1)+2^g=2^g(s-g+2)$.

Note that the subgraph induced by $A$ is a $g$-regular graph. In order to show $F_1$ is a $g$-good-neighbor conditional faulty set and $F_2$ is a $(s-1)$-good-neighbor conditional faulty set of $LeTQ(s,t)$,
we only need to show that $u$ has at least $s-1$ neighbors out of $F_2$ for any $u\in V-F_2$.

Denote $L_1=\{a_{s-1}a_{s-2}\cdots a_1a_00^{t+1}~|~a_i\in \{0,1\},~0\leq i\leq s-1\}$. Then $A\subseteq L_1$ and $F_1-A^*\subseteq L_1$.

First we assume $u\in L_1$. Then we may assume $u=u_{s-1}u_{s-2}\cdots u_1u_00^{t+1}$, and there are $t$ ($t\geq 2$) bits of $u_{s-g-1}\cdots u_1u_0$ equal 1 and the other $s-g-t$ bits equal 0.  If $u$ has a neighbor in $F_1$, then $t=2$. Thus $u=u_{s-1}u_{s-2}\cdots u_{s-g}0^{p}10^{s-g-p-q-2}10^{q+t+1}$ for $p,q\geq 0$ and $p+q\leq s-g-2$. Then $F_1\cap N_V(u)=\{u_{s-1}u_{s-2}\cdots u_{s-g}0^{s-g-q-1}10^{q+t+1},u_{s-1}u_{s-2}\cdots u_{s-g}0^{p}10^{s-g-p+t}\}$. Since $u$ is a vertex of degree $s+1$ by Proposition~\ref{degree}, $u$ has $s-1$ neighbors out of $F_2$.

Then we assume $u\in L_i$ for some $i$ ($2\leq i\leq 2^t$), then all the neighbors of $u$ are out of $F_2$ by Propositions~\ref{LR1} and \ref{LR2}, i.e., $u$  has $s+1$ neighbors out of $F_2$ by Proposition~\ref{degree}.

Now we assume $u\in R_j$ for some $j$ ($1\leq j\leq 2^s$), then $u$ has at most one neighbor in $A^*$ or at most one neighbor in $F_1$ by Proposition~\ref{LR2}. Since $u$ is a vertex of degree $t+1$ by Proposition~\ref{degree}, $u$ has $t\geq s$ neighbors out of $F_2$.

Therefore we complete the proof of Lemma~\ref{g-neighbor}.
\end{proof}

\begin{lemma} \cite{CCYW}
\label{g-01}
For $1\leq s\leq t$ and  $0\leq g\leq 1$, $\kappa^g(LeTQ(s,t))= 2^g(s-g+1)$.
\end{lemma}

\begin{theorem}%Theorem 3.3
\label{g-cut}
For $1\leq s\leq t$ and $0\leq g\leq s$, $\kappa^g(LeTQ(s,t))= 2^g(s-g+1)$.
\end{theorem}

\begin{proof}
First we show that $\kappa^g(LeTQ(s,t))\leq  2^g(s-g+1)$. Let $$A=\{a_{s-1}a_{s-2}\cdots a_{s-g} 0^{s-g+t+1}~|~a_i\in \{0,1\},~s-g\leq i\leq s-1\}\subset V$$ and $F_1=N_{V}(A)$. Clearly, $LeTQ(s,t)-F_1$ is disconnected. By Lemma~\ref{g-neighbor}, $|F_1|=2^g(s-g+1)$ and $F_1$ is a $g$-good-neighbor conditional faulty set. Then $F_1$ is a $R^g$-vertex-cut of $LeTQ(s,t)$. Thus $\kappa^g(LeTQ(s,t))\leq  2^g(s-g+1)$.

Now we show  $\kappa^g(LeTQ(s,t))\geq  2^g(s-g+1)$ by induction on $g$. If $g=0$, the statement holds by Lemma~\ref{g-01}. Assume the induction hypothesis for $g-1$ with $g\geq 1$, that is, $\kappa^{g-1}(LeTQ(s,t))= 2^{g-1}(s-g+2)$. If $s=1$, then $g=1$, by Lemma~\ref{g-01}, $\kappa^1(LeTQ(s,t))=2$ for any $t\geq 1$, the statement is true. Thus we assume $2\leq s\leq t$.

Let $S$ be any $R^g$-vertex-cut of $LeTQ(s,t)$, $X$ the vertex set of a minimum connected component of $LeTQ(s,t)-S$ and $Y$ the set of vertices in $LeTQ(s,t)-S$ not in $X$. By Proposition~\ref{two-sub}, $LeTQ(s, t)$ can be decomposed into two isomorphic subgraphs $LeTQ_{x}^0(s, t)$ and $LeTQ_{x}^1(s, t)$ by fixing $x$, which are isomorphic to $LeTQ(s-1, t)$ if $x\in \{a_0,a_1,\ldots,a_{s-1}\}$, and isomorphic to  $LeTQ(s, t-1)$ if $x\in \{b_0,b_1,\ldots,b_{t-1}\}$. Denote $S_0=S\cap V(LeTQ_{x}^0(s, t))$ and $S_1=S\cap V(LeTQ_{x}^1(s, t))$. Let $A_1,A_2,\cdots A_p$ be the components of $LeTQ_x^0(s, t)-S_0$, and $B_1,B_2,\cdots B_q$ the components of $LeTQ_x^1(s, t)-S_1$, where $p,q\geq 1$.

\vskip 0.2cm
\noindent{\bf Case 1.} There exists $x\in \{a_0,a_1,\ldots,a_{s-1},b_0,b_1,\ldots,b_{t-1}\}$ such that $p=q=1$.
\vskip 0.2cm

In this case, there are no edges between $A_1$ and $B_1$ in $LeTQ(s,t)-S$ as $LeTQ(s,t)-S$ is disconnected. Then all the neighbors of $A_1$ in $LeTQ_{x}^1(s, t)$ belong to $S_1$ and all the neighbors of $B_1$ in $LeTQ_{x}^0(s, t)$ belong to $S_0$. Note that there are $2^{s+t-1}$ independent edges between $LeTQ_{x}^0(s, t)$ and $LeTQ_{x}^1(s, t)$, and hence each edge has at least one end vertex in $S$, then

~~~~~~~$|S|\geq 2^{s+t-1}= 2^g2^{(s-g+1)+(t-2)}\geq 2^g((s-g+1)+(t-2))\geq 2^g(s-g+1)$.

\vskip 0.2cm
\noindent{\bf Case 2.}  There exists $x\in \{a_0,a_1,\ldots,a_{s-1},b_0,b_1,\ldots,b_{t-1}\}$ such that  $p\geq 2$ and $q\geq 2$.
\vskip 0.2cm

In this case, $S_k$ is a vertex-cut of $LeTQ_{x}^k(s, t)$ as $LeTQ_{x}^k(s, t)-S_k$ is disconnected, where $k\in \{0,1\}$.
By Proposition~\ref{two-sub}, each vertex of $LeTQ_{x}^k(s, t)$ has at most one neighbor in $LeTQ_{x}^{\bar{k}}(s, t)$, then $\delta(LeTQ_{x}^k(s, t)-S_k)\geq g-1$ as $\delta(LeTQ(s, t)-S)\geq g$. Thus $S_k$ is a $R^{g-1}$-vertex-cut of $LeTQ_{x}^k(s, t)$. By induction hypothesis, $|S_k|\geq 2^{g-1}(s-g+1)$ if $x\in \{a_0,a_1,\ldots,a_{s-1}\}$ and $|S_k|\geq 2^{g-1}(t-g+1)\geq 2^{g-1}(s-g+1)$ if $x\in \{b_0,b_1,\ldots,b_{t-1}\}$.
Then $|S|=|S_0|+|S_1|\geq 2^g(s-g+1)$.

%\begin{figure}[!htb]\centering{\includegraphics[height=0.23\textwidth]{F4}}Figure 4 Components of $LeTQ_x^0(s, t)-S_0$ and $LeTQ_x^1(s, t)-S_1$\end{figure}

\vskip 0.2cm
\noindent{\bf Case 3.}  For any $x\in \{a_0,a_1,\ldots,a_{s-1},b_0,b_1,\ldots,b_{t-1}\}$, we have $p=1$, $q\geq 2$, or $q=1$, $p\geq 2$.
\vskip 0.2cm
If $p=1$ and $q\geq 2$ for some $x\in \{a_0,a_1,\ldots,a_{s-1},b_0,b_1,\ldots,b_{t-1}\}$, then $S_1$ is a vertex-cut of $LeTQ_{x}^1(s, t)$ as $LeTQ_{x}^1(s, t)-S_1$ is disconnected.
By Proposition~\ref{two-sub}, each vertex of $LeTQ_{x}^1(s, t)$ has at most one neighbor in $LeTQ_{x}^{0}(s, t)$, then $\delta(LeTQ_{x}^1(s, t)-S_1)\geq g-1$.
Thus $S_{1}$ is a $R^{g-1}$-vertex-cut of $LeTQ_{x}^{1}(s, t)$. By induction hypothesis, $|S_1|\geq 2^{g-1}(s-g+1)$ if $x\in \{a_0,a_1,\ldots,a_{s-1}\}$
and $|S_1|\geq 2^{g-1}(t-g+1)\geq 2^{g-1}(s-g+1)$ if $x\in \{b_0,b_1,\ldots,b_{t-1}\}$. If $|S_0 |\geq |S_1|$, then $|S|=|S_0|+|S_{1}|\geq 2|S_1|\geq 2^g(s-g+1)$.
So we assume $|S_0 |< |S_1|$. Then $|A_1|\geq |B_k|$ for any $1\leq k\leq q$. Note that $X$ is the vertex set of a minimum connected component of $LeTQ(s,t)-S$,
and hence $A_1\subseteq Y$, which means $X\subseteq V(LeTQ_x^1(s, t))-S_1$. Recall that $LeTQ_{x}^i(s, t)$ is the subgraph of $LeTQ(s, t)$ by fixing $x=i$ for $i\in \{0,1\}$ and $x\in\{a_0,a_1,\ldots,s_{s-1},b_0,b_1,\ldots,b_{s-1}\}$.  Hence, for any vertex $x_{s-1}\cdots x_1x_0y_{t-1}\cdots y_1y_0 z$ of $X$, $x_i=x=1$ ($0\le i\le s-1$) or $y_j=x=1$ ($0\le j\le t-1$).
By a similar argument, if $q=1$ and $p\geq 2$ for some $x\in \{a_0,a_1,\ldots,a_{s-1},b_0,b_1,\ldots,b_{t-1}\}$, then $X\subseteq V(LeTQ_x^0(s, t))-S_0$, and thus for any vertex $x_{s-1}\cdots x_1x_0y_{t-1}\cdots y_1y_0 z$ of $X$, we have $x_i=x=0$ ($0\le i\le s-1$) or $y_j=x=0$ ($0\le j\le t-1$).
Therefore, for any two vertices $x_{s-1}\cdots x_1x_0y_{t-1}\cdots y_1y_0 z$, $x'_{s-1}\cdots x'_1x'_0 y'_{t-1}\cdots y'_1y'_0 z'$ of $X$, we have $x_i=x'_i$ and $y_j=y'_j$ for all $0\leq i\leq s-1$ and $0\leq j\leq t-1$. Thus $|X|\leq 2$.
Since $g\geq 1$, then $|X|\ge 2$. Thus $|X|=2$ and $X=\{x_{s-1}\cdots x_1x_0y_{t-1}\cdots y_1y_0 0, x_{s-1}\cdots x_1x_0y_{t-1}\cdots y_1y_0 1\}$ for some $x_i,y_j\in \{0,1\}$ ($0\le i\le s-1$, $0\le j\le t-1$), i.e., the subgraph induced by $X$ is a cross edge. Hence $g=1$ and $|S|=s+t\geq 2s=2^g(s-g+1)$.

%Without loss of generality, we assume $p=1$ and $q\geq 2$. Then $S_1$ is a vertex-cut of $LeTQ_{x}^1(s, t)$ as $LeTQ_{x}^1(s, t)-S_1$ is disconnected. By Proposition~\ref{two-sub}, each vertex of $LeTQ_{x}^1(s, t)$ has at most one neighbor in $LeTQ_{x}^{0}(s, t)$, then $\delta(LeTQ_{x}^1(s, t)-S_1)\geq g-1$. Thus $S_{1}$ is a $R^{g-1}$-vertex-cut of $LeTQ_{x}^{1}(s, t)$. By induction hypothesis, $|S_1|\geq 2^{g-1}(s-g+1)$ if $x\in \{a_0,a_1,\ldots,a_{s-1}\}$ and $|S_1|\geq 2^{g-1}(t-g+1)\geq 2^{g-1}(s-g+1)$ if $x\in \{b_0,b_1,\ldots,b_{t-1}\}$. If $|S_0 |\geq |S_1|$, then $|S|=|S_0|+|S_{1}|\geq 2|S_1|\geq 2^g(s-g+1)$. So we assume $|S_0 |< |S_1|$. Then $|A_1|\geq |B_k|$ for any $1\leq k\leq q$. Note that $X$ is the vertex set of a minimum connected component of $LeTQ(s,t)-S$, and hence $A_1\subseteq Y$, which means $X\subseteq V(LeTQ_x^1(s, t))-S_1$.  %By a similar argument, if $q=1$ and $p\geq 2$, then $X\subseteq V(LeTQ_x^0(s, t))-S_0$. Since $g\geq 1$, $|X|\ge 2$. So for any two vertices $x_{s-1}\cdots x_1x_0y_{t-1}\cdots y_1y_0 z, x'_{s-1}\cdots x'_1x'_0 y'_{t-1}\cdots y'_1y'_0 z'$ of $X$, we have $x_i=x'_i$ and $y_j=y'_j$ for $0\leq i\leq s-1$ and $0\leq j\leq t-1$. Then $|X|\leq 2$, i.e., $|X|=2$, and then the subgraph induced by $X$ is a cross edge. Thus $g=1$ and $|S|=s+t\geq 2s=2^g(s-g+1)$.

Therefore we complete the proof of Theorem~\ref{g-cut}.
\end{proof}

\section{The $g$-good-neighbor conditional diagnosability of $LeTQ(s, t)$}

In this section, first we will give some lemmas, then determine $t_g(LeTQ(s, t))$ for $1\leq s\leq t$ and $0\leq g\leq s$ under the PMC model and MM$^*$ model, respectively.

%%%%%%%%%%%%

\begin{theorem}
\label{geq}
For $1\leq s\leq t$ and $0\leq g\leq s$, $t_g(LeTQ(s,t))\leq 2^g(s-g+2)-1$ under the PMC model and MM$^*$ model, respectively.
\end{theorem}

\begin{proof} Let  $A=\{a_{s-1}a_{s-2}\cdots a_{s-g} 0^{s-g+t+1}~|~a_i\in \{0,1\},~s-g\leq i\leq s-1\}\subset V$, and let $F_1=N_{V}(A)$, $F_2=N_{V}[A]$.
Then by Lemma~\ref{g-neighbor}, $|F_2|=2^g(s-g+2)$ and $F_i$ is a $g$-good neighbor conditional faulty set for $i=1,2$. Note that $F_1\triangle F_2=A$ and $N_{V}(A)=F_1$, and hence $(F_1, F_2)$ is an indistinguishable pair under the PMC model by Proposition~\ref{PMC}, and  under the MM$^*$ model by Proposition~\ref{MM}. From the definition of $t_g$, $t_g(LeTQ(s,t))\leq 2^g(s-g+2)-1$ under the PMC model and MM$^*$ model, respectively.
\end{proof}

%%%%%%%%%%%

\begin{lemma}
\label{two}
Let $F_1$ and $F_2$ be two distinct $g$-good-neighbor conditional faulty sets in $LeTQ(s,t)$ such that $|F_1|, |F_2| \leq 2^g(s-g+2)-1$. For $1\leq s\leq t$ and $0\leq g\leq s$, we have $F_1\cup F_2 \neq V$.
\end{lemma}

\begin{proof} Since $s-g\geq 0$ and $t\geq 1$, we have $2^{s-g+t}\geq (s-g+t)+1\geq s-g+2$. Then
$$|F_1\cup F_2|\leq |F_1|+|F_2|\leq 2^{g+1}(s-g+2)-2\leq 2^{g+1}2^{s-g+t}-2<2^{s+t+1}.$$ Thus $F_1\cup F_2 \neq V$.
\end{proof}

\subsection{$t_g(LeTQ(s,t))$ under the PMC model}

In this subsection, we will determine $t_g(LeTQ(s,t))$ for $1\leq s\leq t$ and $0\leq g\leq s$ under the PMC model.

\begin{lemma}
\label{leq1}
Let $1\leq s\leq t$ and $0\leq g\leq s$. For any two distinct $g$-good-neighbor conditional faulty sets $F_1$ and $F_2$ in $LeTQ(s,t)$ with $|F_1|, |F_2| \leq 2^g(s-g+2)-1$, $(F_1, F_2)$ is a distinguishable pair under the PMC model.
\end{lemma}

\begin{proof}
Suppose that $(F_1, F_2)$ is an indistinguishable pair. Assume, without loss of generality, that $F_2-F_1\neq \emptyset$.
By Lemma~\ref{two}, we have $F_1\cup F_2 \neq V$. Since $F_1\neq F_2$, $F_1\triangle F_2\neq \emptyset$. By Proposition~\ref{PMC}, $E_{LeTQ(s,t)}(F_1\triangle F_2,V-F_1-F_2)=\emptyset$. Since  $LeTQ(s,t)$ is connected, $F_1\cap F_2\neq \emptyset$,  and thus $F_1\cap F_2$ is a vertex cut of $LeTQ(s,t)$. Since $F_1$ and $F_2$ are $g$-good-neighbor conditional faulty sets of $LeTQ(s,t)$, $F_1\cap F_2$ is also a $g$-good-neighbor conditional faulty set, which implies $F_1\cap F_2$ is a $R^g$-vertex-cut of $LeTQ(s,t)$. By Theorem~\ref{g-cut}, $|F_1\cap F_2|\geq 2^g(s-g+1)$.

Since $F_1$ is a $g$-good-neighbor conditional faulty set, all the vertices in $F_2-F_1$ have at least $g$ neighbors out of $F_1$. By  Proposition~\ref{PMC}, $E_{LeTQ(s,t)}(F_1\triangle F_2,V-F_1-F_2)=\emptyset$, then $\delta(LeTQ(s,t)[F_2-F_1])\geq g$. Thus $|F_2-F_1|\geq 2^g$ by Proposition~\ref{LeTQ-neighbor}. So
$$2^g(s-g+2)-1\geq |F_2| = |F_1\cap F_2| + |F_2-F_1| \geq 2^g(s-g+1)+2^g=2^g(s-g+2),$$
a contradiction.
\end{proof}

By Theorem~\ref{geq} and Lemma~\ref{leq1}, we get the $g$-good neighbor of conditional diagnosability of $LeTQ(s,t)$ under the PMC model for $1\leq s\leq t$ and $0\leq g\leq s$.

\begin{theorem}
\label{PMC1}
For $1\leq s\leq t$ and $0\leq g\leq s$, $t_g(LeTQ(s,t))=2^g(s-g+2)-1$ under the PMC model.
\end{theorem}

\subsection{$t_g(LeTQ(s,t))$ under the MM$^*$ model}

In this subsection, we will determine $t_g(LeTQ(s,t))$ for $1\leq s\leq t$ and $0\leq g\leq s$ under the MM$^*$ model.

\begin{lemma}
\label{F1capF2}
For any two distinct faulty sets $F_1,F_2$ in $LeTQ(s,t)$ with $F_1\cup F_2 \neq V$, if $(F_1,F_2)$ is an indistinguishable pair under the MM$^*$ model and $LeTQ(s,t)-F_1-F_2$ has no isolated vertices, then $F_1\cap F_2\neq \emptyset$.
\end{lemma}

\begin{proof} Suppose to the contrary that $F_1\cap F_2=\emptyset$. Let $H$ be the component of $LeTQ(s,t)-F_1-F_2$, then $\delta(H)\ge 1$. By Proposition~\ref{MM}(1), $E_{LeTQ(s,t)}(F_i,V(H))=\emptyset$ for $i=1,2$ as $(F_1,F_2)$ is an indistinguishable pair. So $E_{LeTQ(s,t)}(F_1\cup F_2,V-F_1-F_2)=\emptyset$, i.e., $LeTQ(s,t)$ is disconnected, a contradiction. \end{proof}

For $g=0$,  the $0$-good-neighbor condition does not have any restriction on the faulty sets in this case, then $t_0(LeTQ(s,t))=t(LeTQ(s,t))$.

\begin{theorem}
\label{s=t=1}
For $0\leq g\leq 1$, $t_g(LeTQ(1,1))=1$ under the MM$^*$ model.
\end{theorem}

\begin{proof}
For $s=t=1$, $LeTQ(1,1)$ is a 8-cycle (see Figure 1).  Let $F'_1=\{000,110\}$ and $F'_2=\{101,011\}$. Then $|F'_i|=2$ and $\delta(LeTQ(1,1)-F'_i)=1$ for $1\leq i\leq 2$. By Proposition~\ref{MM}, $(F'_1,F'_2)$ is an indistinguishable pair. Thus $t_g(LeTQ(1,1))\leq 1$ by the definition of $t_g(LeTQ(1,1))$. On the other hand, it is easy to check that for any two distinct faulty sets $F'_1,F'_2$ in $LeTQ(1,1)$ with $|F'_1|=|F'_2|=1$, $(F_1,F_2)$ is a distinguishable pair. Then $t_g(LeTQ(1,1))\geq 1$. Hence $t_g(LeTQ(1,1))=1$.
\end{proof}

\begin{theorem}
\label{g=0}
For $1\leq s\leq t$ with $s+t\geq 3$, $t_0(LeTQ(s,t))=s+1$ under the MM$^*$ model.
\end{theorem}

\begin{proof}
First we show that $t_0(LeTQ(s,t))\leq s+1$. Let $v\in V(L)$, $F'_1=N_{V}(v)$ and $F'_2=N_{V}[v]$. Note that $F'_1\triangle F'_2=\{v\}$, then $(F'_1,F'_2)$ is an indistinguishable pair by Proposition~\ref{MM}. Since $|F'_1|=s+1$ and $|F'_2|=s+2$, $t_0(LeTQ(s,t))\leq s+1$ by the definition of $t_0(LeTQ(s,t))$.

Now we show that $t_0(LeTQ(s,t)\geq s+1$. That is, for any two distinct faulty sets $F_1,F_2$ in $LeTQ(s,t)$ with $|F_1|,|F_2|\leq s+1$, $(F_1,F_2)$ is a distinguishable pair. Suppose to the contrary that  $(F_1,F_2)$ is an indistinguishable pair. Note that $|F_1\cup F_2|\leq |F_1|+|F_2|\leq 2(s+1)<2(s+t+1)<2^{s+t+1}$, and hence $F_1\cup F_2 \neq V$.

\vskip 0.2cm
\noindent{\bf Claim 1.} $LeTQ(s,t)-F_1-F_2$ has no isolated vertices.
\vskip 0.2cm

\vskip 0.2cm\noindent{\bf Proof of Claim 1.} Let $W$ be the set of isolated vertices in $LeTQ(s,t)-F_1-F_2$, and $H$ the subgraph induced by the vertex set $V-F_1-F_2-W$. Suppose to the contrary that $W\neq \emptyset$ and let $w\in W$.

If $V(H)\neq \emptyset$, then $E_{LeTQ(s,t)}(F_1\triangle F_2, V(H))=\emptyset$ by Proposition~\ref{MM}(1). Since $LeTQ(s,t)$ is connected, then $F_1\cap F_2\neq \emptyset$ and thus $F_1\cap F_2$ is a vertex-cut of $LeTQ(s,t)$. By Theorem~\ref{g-cut},
$|F_1\cap F_2|\geq s+1$. That is $|F_1|=|F_2|=s+1$ and $F_1=F_2$, a contradiction. Therefore $V(H)=\emptyset$.

Note that $|W|=|V|-|F_1\cup F_2|\geq 2^{s+t+1}-2(s+1)>2(t+1)$. Thus
$$2(s+1)(t+1)< (s+1)|W|\leq \sum_{w\in W}d_{V}(w)\leq \sum_{v\in F_1\cup F_2}d_{V}(v)\leq 2(s+1)(t+1),$$ a contradiction. \q

By Claim 1, $LeTQ(s,t)-F_1-F_2$ has no isolated vertices.
Then $E_{LeTQ(s,t)}(F_1\triangle F_2,V-F_1-F_2)=\emptyset$ by Proposition~\ref{MM}(1). Since  $LeTQ(s,t)$ is connected, $F_1\cap F_2\neq \emptyset$,  and thus $F_1\cap F_2$ is a vertex cut of $LeTQ(s,t)$. By Theorem~\ref{g-cut}, $|F_1\cap F_2|\geq s+1$.
Note that $|F_1|,|F_2|\leq s+1$, then $|F_1|=|F_2|=s+1$ and $F_1=F_2$, a contradiction.

Therefore, we complete the proof of Theorem~\ref{g=0}.
\end{proof}

\begin{figure}[!htb]
\centering
{\includegraphics[height=0.25\textwidth]{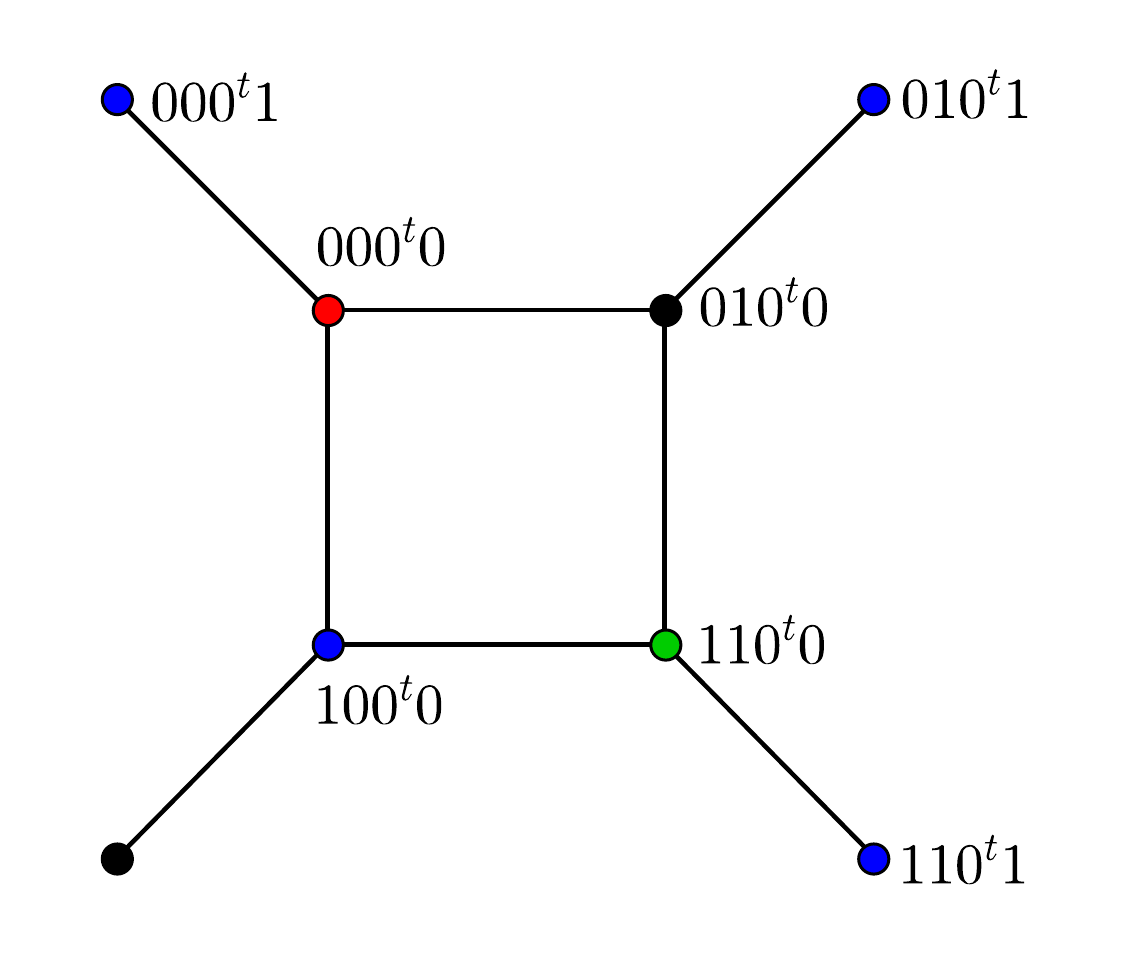}}

Figure 4 An indistinguishable pair $(F_1',F_2')$
\end{figure}

\begin{theorem}
\label{s=2}
For $t\geq 2$, $t_1(LeTQ(2,t))=4$ under the MM$^*$ model.
\end{theorem}

\begin{proof}
First we show that $t_1(LeTQ(2,t))\leq 4$. Let $F_1'=\{000^t0,000^t1,110^t1,100^t1,010^t1\}$ and $F_2'=\{110^t0,000^t1,110^t1,100^t1,010^t1\}$  (see Figure 4). Then $F_1'\triangle F_2'=\{000^t0,110^t0\}$. It is easy to certify that $\delta(LeTQ(s,t)-F'_i)=1$ for $1\leq i\leq 2$ and $(F_1',F_2')$ is an indistinguishable pair, thus $t_1(LeTQ(2,t))\leq 4$ by the definition of $t_1(LeTQ(2,t))$.

Now we show that $t_1(LeTQ(2,t))\geq 4$. That is, for any two distinct 1-good-neighbor conditional faulty sets $F_1,F_2$ in $LeTQ(2,t)$ with $|F_1|,|F_2|\leq 4$, $(F_1,F_2)$ is a distinguishable pair. Suppose to the contrary that  $(F_1,F_2)$ is an indistinguishable pair.

\vskip 0.2cm
\noindent{\bf Claim 2.} $LeTQ(2,t)-F_1-F_2$ has no isolated vertices.
\vskip 0.2cm

\noindent{\bf Proof of Claim 2.} Let $W$ be the set of isolated vertices in $LeTQ(2,t)-F_1-F_2$, and $H$ the subgraph induced by the vertex set $V-F_1-F_2-W$. We will show that $W=\emptyset$. Suppose to the contrary that $W\neq \emptyset$ and let $w\in W$. Then $N_{V-F_1}(w)\subseteq F_2-F_1$.

Since $F_1$ is a 1-good-neighbor conditional faulty set, $w$ must have at least one neighbor out of $F_1$. Note that $F_1$ and $F_2$ are indistinguishable, and hence $|N_{V-F_1}(w)| =1$ by Proposition~\ref{MM}(3). Similarly, $N_{V-F_2}(w)\subseteq F_1-F_2$ and $|N_{V-F_2}(w)| =1$. Therefore $|N_{F_1\cap F_2}(w)| \ge 1$ by Proposition~\ref{degree}. So we have
\begin{eqnarray*}|W|&\leq& \sum_{w\in W}|N_{F_1\cap F_2}(w)|\leq \sum_{v\in F_1\cap F_2}d_{V}(v)\leq (t+1)|F_1\cap F_2|\\
&\leq& (t+1)(|F_1|-1)\leq 3(t+1).
\end{eqnarray*}

If $V(H)=\emptyset$, then $|V |=|F_1 \cup F_2|+|W|$. Note that $2^{t+3}>4(t+3)$, then
\begin{eqnarray*}8&\geq& |F_1|+|F_2|=|F_1 \cup F_2|+|F_1 \cap F_2|=|V |-|W|+|F_1 \cap F_2|\\ &\geq& 2^{3+t}-(3t+4)+1>4(t+3)-3t-4+1=t+9,
\end{eqnarray*}
a contradiction.

So $V(H)\neq \emptyset$. Since $F_1$ and $F_2$ are indistinguishable, $E_{LeTQ(2,t)}(F_1\triangle F_2,V(H))=\emptyset$ by Proposition~\ref{MM}(1). Since  $LeTQ(2,t)$ is connected, $F_1\cap F_2\neq \emptyset$ is a vertex cut of $LeTQ(2,t)$. Since $F_1$ and $F_2$ are $1$-good-neighbor conditional faulty sets, $F_1\cap F_2$ is also a $1$-good-neighbor conditional faulty set, which implies $F_1\cap F_2$ is a $R^1$-vertex-cut of $LeTQ(2,t)$.
By Theorem~\ref{g-cut}, $|F_1\cap F_2|\geq 4$. Recall that $|F_1|,|F_2|\leq 4$, then $F_1=F_2$, a contradiction.
\q

By Claim 2, $LeTQ(2,t)-F_1-F_2$ has no isolated vertices. Then by Lemma~\ref{F1capF2}, $F_1\cap F_2\neq\emptyset$. Thus $F_1\cap F_2$ is a vertex cut of $LeTQ(2,t)$ by Proposition~\ref{MM}(1). Since $F_1$ and $F_2$ are $1$-good-neighbor conditional faulty sets, $F_1\cap F_2$ is also a $1$-good-neighbor conditional faulty set, which implies $F_1\cap F_2$ is a $R^1$-vertex-cut of $LeTQ(2,t)$. By Theorem~\ref{g-cut}, $|F_1\cap F_2|\geq 4$.  Recall that  $|F_1|,|F_2|\leq 4$, we have $F_1=F_2$, a contradiction.

Therefore, we complete the proof of Theorem~\ref{s=2}.
\end{proof}

\begin{theorem}
\label{g>1}
For $s=1$, $t\geq 2$ and $g=1$, or $3\leq s\leq t$ and $g=1$, or $2\leq s\leq t$ and $2\leq g\leq s$, $t_g(LeTQ(s,t))=2^g(s-g+2)-1$ under the MM$^*$ model.
\end{theorem}

\begin{proof}
By Theorem~\ref{geq}, $t_g(LeTQ(s,t))\leq 2^g(s-g+2)-1$. In the following, we show that $t_g(LeTQ(s,t))\geq 2^g(s-g+2)-1$. That is, for any two
distinct $g$-good-neighbor conditional faulty sets $F_1$ and $F_2$ with $|F_1|, |F_2| \leq 2^g(s-g+2)-1$, we show that $(F_1, F_2)$ is a distinguishable pair. Suppose to the contrary that  $(F_1,F_2)$ is an indistinguishable pair. By Lemma~\ref{two}, we have $F_1\cup F_2 \neq V$. Since $F_1\neq F_2$, then $F_1\triangle F_2\neq \emptyset$.

\vskip 0.2cm
\noindent{\bf Claim 3.} $LeTQ(1,t)-F_1-F_2$ has no isolated vertices.
\vskip 0.2cm

\noindent{\bf Proof of Claim 3.} Let $W$ be the set of isolated vertices in $LeTQ(s,t)-F_1-F_2$, and $H$ the subgraph induced by the vertex set $V-F_1-F_2-W$. We will show that $W=\emptyset$. Suppose to the contrary that $W\neq \emptyset$ and let $w\in W$. Then $N_{V-F_1}(w)\subseteq F_2-F_1$.

Since $F_1$ is a $g$-good-neighbor conditional faulty set, then $|N_{V-F_1}(w)|\geq g$.
Note that $F_1$ and $F_2$ are indistinguishable, then $|N_{V-F_1}(w)| =1$ by Proposition~\ref{MM}(3). Hence $g=1$ and $|F_1|, |F_2| \leq 2s+1$. Similarly, $N_{V-F_2}(w)\subseteq F_1-F_2$ and $|N_{V-F_2}(w)| =1$. Therefore $|N_{F_1\cap F_2}(w)| \ge s-1$ by Proposition~\ref{degree}.

If $V(H)=\emptyset$, then $|V|=|F_1 \cup F_2|+|W|$. Note that $2^{s+t+1}\geq 4(s+t+1)$ as $s+t+1\geq 4$, then $|W|\geq 2^{s+t+1}-2(2s+1)\geq 4(s+t+1)-2(2s+1)>2(t+1)$. Thus
\begin{eqnarray*}2(s+1)(t+1)<(s+1)|W|\leq \sum_{w\in W}d_{V}(v)\leq \sum_{v\in F_1\cup F_2}d_V(v)\leq (t+1)2(s+1),
\end{eqnarray*}
a contradiction.

So $V(H)\neq \emptyset$. Since $F_1$ and $F_2$ are indistinguishable, then $E_{LeTQ(s,t)}(F_1\triangle F_2,V(H))=\emptyset$ by Proposition~\ref{MM}(1).
Thus $F_1\cap F_2\neq \emptyset$ is a vertex-cut of $LeTQ(s,t)$. By Theorem~\ref{g-cut}, $|F_1\cap F_2|\geq 2s$. Note that for any $w \in W$, $w$ has one neighbor in $F_1-F_2$ and one neighbor in $F_2-F_1$. So we have $|F_1\cap F_2| =2s$, which implies $|F_1| =|F_2| =2s+1$ and $|F_1-F_2| =|F_2-F_1| =1$.  Assume $F_1-F_2=\{v_1\}$ and $F_2-F_1=\{v_2\}$. Then for any $w \in W$, $v_1, v_2\in N_{V}(w)$. By Proposition~\ref{common neighbor}, $|W|\leq 2$, and $LeTQ(s,t)$ has no triangles, then $w$ and $v_i$ have no common neighbors for $i=1,2$.

\vskip 0.2cm
\noindent{\bf Case 1.} $|W|=1$.
\vskip 0.2cm

In this case $W=\{w\}$ and $|N_{F_1 \cap F_2}(w)|\ge s-1$, $|N_{F_1 \cap F_2}(v_1)\ge s$ and $|N_{F_1 \cap F_2}(v_2)|\ge s$ by Proposition~\ref{degree}.  By Proposition~\ref{common neighbor}, $v_1$ and $v_2$ have at most one common neighbor in $F_1\cap F_2$. Then $$2s=|F_1\cap F_2|\geq |E_{LeTQ(s,t)}(F_1\cap F_2,\{w,v_1,v_2\})|-1\geq (s-1)+s+(s-1)=3s-2,$$ i.e., $s\leq 2$. Therefore  $s=1$ and $t\geq 2$. Set $F_1\cap F_2=\{u_1,u_2\}$. If $w\in L$, we may assume that $v_2\in R$ as each vertex in $L$ has exactly one neighbor in $R$ by Proposition~\ref{LR2}. Then $d_V(v_2)=t+1$ by Proposition~\ref{degree}. Since $(N_{V}(v_2)\setminus \{w\})\subseteq F_1 \cap F_2$ and $|F_1 \cap F_2|=2$, then $t=2$  and $v_2u_1,v_2u_2\in E(LeTQ(1,2))$. Note that $d_V(v_1)\geq 2$, we may assume $v_1u_1\in E(LeTQ(1,2))$. Then $LeTQ(1,2)$ contains a four cycle $wv_1u_1v_2$ with $w\in L$, which is impossible by Figure 1. So $w\in R$. Assume that $wu_1\in E(LeTQ(1,t))$. Then $v_1u_2,v_2u_2\in E(LeTQ(1,t))$ and $wu_2\notin E(LeTQ(1,t))$ as $LeTQ(1,t)$ has no triangles. Thus $t=2$. Hence $LeTQ(1,2)$ contains a four cycle $wv_1u_2v_2$ with $v_1,v_2\in L$, which is impossible by Figure 1.

\vskip 0.2cm
\noindent{\bf Case 2.} $|W|=2$.
\vskip 0.2cm

Denote $W=\{w_1,w_2\}$. Then $|N_{F_1 \cup F_2}(w_i)|\ge s-1$ for $1\leq i\leq 2$, $|N_{F_1 \cap F_2}(v_1)|\ge s-1$ and $|N_{F_1 \cup F_2}(v_2)|\ge s-1$ by Proposition~\ref{degree}.  By Proposition~\ref{common neighbor}, $v_1$ and $v_2$ have no common neighbors in $F_1\cup F_2$, $w_1$ and $w_2$ have no common neighbors in $F_1\cup F_2$. Then  $$2s=|F_1\cap F_2|\geq |E_{LeTQ(s,t)}(F_1\cap F_2,\{w_1,w_2,v_1,v_2\})|\geq 4(s-1)=4s-4,$$
i.e., $s\leq 2$. Therefore $s=1$ and $t\geq 2$. Note that $v_1w_1v_2w_2$ is a four cycle of $LeTQ(1,t)$. By Proposition~\ref{LR2}, each vertex in $L$ has exactly one neighbor in $R$ and vice visa. Then $|\{v_1,w_1,v_2,w_2\}\cap L|=4$, or 2, or 0. By Proposition~\ref{LR1}, the subgraph induced by $L$ are disjoint copies of $LTQ_1\cong K_2$, then $|\{v_1,w_1,v_2,w_2\}\cap L|\neq 4$.  If $|\{v_1,w_1,v_2,w_2\}\cap L|=2$, then $t=2$ as $|F_1\cap F_2|=2$. Thus $LeTQ(1,2)$ contains a four cycle $v_1w_1v_2w_2$ with $|\{v_1,w_1,v_2,w_2\}\cap L|=2$, which is impossible by Figure 1. If $|\{v_1,w_1,v_2,w_2\}\cap L|=0$, then $d_V(x)=t+1\geq 3$ for $x\in \{v_1,w_1,v_2,w_2\}$. Thus $|F_1\cap F_2|\geq 4$, a contradiction.
\q

By Claim 3, $LeTQ(s,t)-F_1-F_2$ has no isolated vertices. Then by Lemma~\ref{F1capF2}, $F_1\cap F_2\neq\emptyset$. Since $F_1$ and $F_2$ are indistinguishable, by Proposition~\ref{MM}(1), $F_1\cap F_2$ is a vertex cut of $LeTQ(s,t)$. Note that $F_1\cap F_2$ is also a $g$-good-neighbor conditional faulty set, and hence $F_1\cap F_2$ is a $R^g$-vertex-cut of $LeTQ(s,t)$. By Theorem~\ref{g-cut}, $|F_1\cap F_2|\geq 2^g(s-g+1)$.

Since $F_1$ is a $g$-good-neighbor conditional faulty set, all the vertices in $F_2-F_1$ have at least $g$ neighbors out of $F_1$.
By Proposition~\ref{MM}(1), $E_{LeTQ(s,t)}(F_1\triangle F_2,V-F_1-F_2)=\emptyset$, then $\delta(LeTQ(s,t)[F_2-F_1])\geq g$. Thus $|F_2-F_1|\geq 2^g$ by Proposition~\ref{LeTQ-neighbor}. So
$$2^g(s-g+2)-1\geq |F_2| = |F_1\cap F_2| + |F_2-F_1| \geq 2^g(s-g+1)+2^g=2^g(s-g+2),$$
a contradiction.

Therefore, we complete the proof of Theorem~\ref{g>1}.
\end{proof}

\section{Conclusions}

In this paper, we consider the $g$-good-neighbor conditional diagnosability of the
locally exchanged twisted cube $LeTQ(s,t)$ under the PMC model and MM$^*$ model, respectively. We
show that when $1\leq s\leq t$ and $0\leq g\leq s$, the $g$-good-neighbor conditional diagnosability of $LeTQ(s,t)$ under
the PMC model is $t_g(LeTQ(s,t))=2^g(s-g+2)-1$.
When $1\leq s\leq t$ and $0\leq g\leq s$, the $g$-good-neighbor conditional diagnosability of $LeTQ(s,t)$ under the
MM$^*$ model is $t_g(LeTQ(s,t))=2^g(s-g+2)-\epsilon$, where $\epsilon=3$ if $g=1$ and $s=t=1$, $\epsilon=2$ if $g=0$ and $s=t=1$, or $g=1$ and $s=2$, and $\epsilon=1$ otherwise.

Compared with the conventional diagnosability, the $g$-good-neighbor conditional diagnosability improves accuracy in measuring the reliability of interconnection networks in heterogeneous environments. Future research on this topic will involve studying the $g$-good-neighbor conditional
diagnosability of other interconnection networks.

\vskip 0.6cm

\noindent{\bf \Large Acknowledgments}

\vskip 0.3cm

\noindent Huiqing Liu is partially supported by NNSFC under grant numbers 11571096 and 61373019. Xiaolan Hu is partially supported
by NNSFC under grant number 11601176, NSF of Hubei Province under grant number 2016CFB146, and self-determined research funds of CCNU from the colleges' basic research and operation of MOE.

\end{document}